\documentclass[a4paper]{article}
\usepackage[mathscr]{eucal}
\usepackage{amssymb}
\usepackage{latexsym}
\usepackage{amsthm}
\usepackage{amsmath}
\usepackage[dvips]{graphicx}
\usepackage{psfrag}

\DeclareMathOperator{\diag}{diag}

\newtheorem{theorem}{Theorem}[section]

\newtheorem{lemma}[theorem]{Lemma}

\theoremstyle{definition}

\newtheorem{remark}[theorem]{Remark}

\newtheorem{example}[theorem]{Example}

\makeatletter
\@addtoreset{equation}{section}

\makeatother

\title{A generalization of Larman--Rogers--Seidel's theorem}
\author{Hiroshi Nozaki}

\begin{document}
\maketitle
\begin{center}
	Graduate School of Information Sciences,
	Tohoku University \\
	Aramaki-Aza-Aoba 6-3-09,
	Aoba-ku,
	Sendai 980-8579,
	Japan\\ 
	nozaki@ims.is.tohoku.ac.jp
\end{center}
\renewcommand{\thefootnote}{\fnsymbol{footnote}}

\begin{abstract}
A finite set $X$ in the $d$-dimensional Euclidean space is called an $s$-distance set 
if the set of Euclidean distances between any two distinct points of $X$ has size $s$. 
Larman--Rogers--Seidel proved that if the cardinality of a two-distance set is greater than $2d+3$, then there 
exists an integer $k$ such that $a^2/b^2=(k-1)/k$, where $a$ and $b$ are the distances. 
In this paper,  we give an extension of this theorem for any $s$. Namely, if the size of an $s$-distance set is greater 
than some value depending on $d$ and $s$, then certain functions of $s$ distances become integers.  
Moreover, we prove that if the size of $X$ is greater than the value, then the number of $s$-distance sets is finite. 
\end{abstract}

\section{Introduction} 
Let $\mathbb{R}^d$ be the $d$-dimensional Euclidean space. 
A subset $X$ in $\mathbb{R}^d$ is called an $s$-distance set, if $|A(X)|=s$ where $A(X)=\{d(x,y) \mid x,y 
\in X, x\ne y\}$, $d(x,y)$ is the Euclidean distance of $x$ and $y$, and $|\ast|$ denotes the cardinality. 
A basic problem is to determine the maximum possible cardinality of $s$-distance 
sets in $\mathbb{R}^d$ or in the unit sphere $S^{d-1}$ for a fixed $s$. 

An absolute upper bound for the cardinality of an $s$-distance set in $\mathbb{R}^d$ was 
given by Bannai, Bannai and Stanton, {\it i.e.} $|X| \leq \binom{d+s}{s}$ \cite{Bannai-Bannai-Stanton} (for $s=2$, the bound 
was proved in \cite{Blokhuis}). 
We also have the bound $|X|\leq \binom{d+s-1}{s}+\binom{d+s-2}{s-1}$ for an $s$-distance set in $S^{d-1}$ 
\cite{Delsarte-Goethals-Seidel}. 

 Larman, Rogers and Seidel \cite{Larman-Rogers-Seidel} proved a very useful theorem to determine the maximum cardinality of 
two-distance sets. Namely, if the cardinality of a two-distance set $X$ in 
$\mathbb{R}^d$ is at least $2d+4$, then there exists a positive integer $k$ such that 
$\alpha_1^2/\alpha_2^2=(k-1)/k$ where $A(X)=\{\alpha_1,\alpha_2\}$. 
Moreover, the integer $k$ is bounded 
above by $1/2+\sqrt{d/2}$. 
The condition $|X| \geq 2d+4$ was improved to $|X|\geq 2d+2$ in \cite{Neumaier}. 
There exists a $(2d+1)$-point two-distance set whose $\alpha_1^2/\alpha_2^2$ is not an integer.
Indeed, they are obtained from 
the spherical embedding of the conference graph \cite{Neumaier}. 
Since we may assume $\alpha_2=1$ for a two-distance set in $\mathbb{R}^d$, the distance $\alpha_1$ is determined by an integer $k$. 

The maximum cardinalities of two-distance 
sets were determined for $d\leq 8$ \cite{Croft, Kelly, Lisonek}. 
Larman--Rogers--Seidel's theorem helped the classification of maximum two-distance sets for $d\leq 7$ in \cite{Lisonek}.

Musin \cite{Musin} gave a certain general method to improve the known upper bounds
 for spherical two-distance sets. 
Let $\beta_1,\beta_2$ be the inner products between distinct points of a spherical two-distance set, and $k$ be the ratio of 
Larman--Rogers--Seidel's theorem. 
In his method, one inner product $\beta_1$ is expressed by a certain function of an integer $k$ and the other inner product $\beta_2$ 
by Larman--Rogers--Seidel's theorem. 
This is one of the key ideas in \cite{Musin}, and the maximum cardinalities 
of two-distance sets in $S^{d-1}$ were determined for $7 \leq d \leq 21$ and $24 \leq d \leq 39$. 
 
Larman--Rogers--Seidel's theorem can be expected a lot of applications, and 
its extension for any $s$ is one of most important problems in the theory 
of few distance sets. 

In the present paper, we give a generalization of Larman--Rogers--Seidel's theorem for any $s$. 
Namely, if the 
cardinality of an $s$-distance set $X\subset \mathbb{R}^d$ is at least $ 2 \binom{d+s-1}{s-1}+ 2 \binom{d+s-2}{s-2}$, then 
\[
\prod_{i=1,2,\ldots,s,i \ne j} \frac{\alpha_j^2}{\alpha_j^2-\alpha_i^2}
\] is an integer $k_i$ for each $i=1,2, \ldots, s$, where 
$A(X)=\{ \alpha_1, \alpha_2, \ldots ,\alpha_s\}$. Moreover, $|k_i|$ is bounded by a certain function of $s$ and 
$d$. 
For $s=2$, $k_1=\alpha_2^2/(\alpha_2^2-\alpha_1^2)$ is an integer, and we can transform the equality to $\alpha_1^2/\alpha_2^2=(k_1-1)/k_1$. This theorem is exactly an extension of Larman--Rogers--Seidel's theorem. 
Furthermore we show that the distances $\alpha_i$ ($\alpha_s=1$) are uniquely determined from 
given integers $k_i$.  

A problem about the finiteness of the number of $s$-distance sets is also interesting. 
Einhorn--Schoenberg \cite{Einhorn-Schoenberg} proved that there are finitely many two-distance sets $X$ with $|X|\geq d+2$. 
Actually we have infinitely many two-distance sets $X$ in $\mathbb{R}^d$ with $|X|=d+1$ \cite{Einhorn-Schoenberg}.
In Section $4$, we prove that there are finitely many $s$-distance sets $X$ in $\mathbb{R}^d$ 
with $|X|\geq 2\binom{d+s-1}{s-1}+2\binom{d+s-2}{s-2}$. 
The result of Einhorn--Schoenberg is much better for $s=2$, but our result is new for $s > 2$.

In Section $5$, we show a generalization of Larman--Rogers--Seidel's theorem for spherical $s$-distance sets. 
The statement is the same as the Euclidean case except for the condition $|X|\geq 2\binom{d+s-2}{s-1}+2 \binom{d+s-3}{s-2}$. 
The inequality is not restrictive in comparison with the Euclidean case.   

A finite $X\subset \mathbb{R}^d$ is said to be antipodal if $-x \in X$ for every $x \in X$. 
The cardinality of an antipodal $s$-distance set in $S^{d-1}$ is bounded above by $2\binom{d+s-2}{s-1}$ 
\cite{Delsarte-Goethals-Seidel2, Delsarte-Goethals-Seidel}. 
Antipodal spherical $s$-distance sets are closely related with spherical $t$-designs \cite{Delsarte-Goethals-Seidel} 
or Euclidean lattices. 
Indeed, a tight (or minimal) spherical $(2s-1)$-design becomes an antipodal $s$-distance set 
whose size is $2 \binom{d+s-2}{s-1}$ \cite{Delsarte-Goethals-Seidel}. 
In Section 5, we prove that if an antipodal spherical $s$-distance set $X$ has sufficiently large size, 
then every usual inner product of any distinct two points in $X$ is a rational number.

\section{Preliminaries}
We prepare some notation and results.
Let ${\rm Hom}_l(\mathbb{R}^d)$ be the linear space of all real homogeneous polynomials of degree $l$, with $d$ variables.  We define $P_l(\mathbb{R}^d):=\sum_{i=0}^l {\rm Hom_i(\mathbb{R}^d)}$ and $P_l^{\ast}(\mathbb{R}^d):=\sum_{i=0}^{\lfloor l/2 \rfloor} {\rm Hom}_{l-2i}(\mathbb{R}^d)$. 
 Let $\xi_1$, $\xi_2$, $\ldots$, $\xi_d$ be independent variables, and let $\xi_0= \xi_1^2+ \xi_2^2+ \cdots +\xi_d^2$.  We define $W_l(\mathbb{R}^d)$ to be the linear space spanned by the monomials $\xi_0^{\lambda_0}\xi_1^{\lambda_1} \cdots \xi_d^{\lambda_d}$ with $\lambda_0+\lambda_1+ \cdots +\lambda_d \leq l$ and $\lambda_i \geq0$. Let $P_l(X)$, $P_l^{\ast}(X)$ and $W_l(X)$ be the linear spaces of all functions which are the restrictions of
the corresponding polynomials to $X \subset \mathbb{R}^d$. Then we know the dimensions of the following linear spaces. 
	\begin{theorem}[\cite{Bannai-Bannai-1, Delsarte-Goethals-Seidel}] \quad \\
	{\rm (i)} $\dim P_l(\mathbb{R}^d)= \binom{d+l}{l}$.\\
	{\rm (ii)} $\dim P_l(S^{d-1})= \binom{d+l-1}{l}+\binom{d+l-2}{l-1}$. \\ 
	{\rm (iii)} $\dim P_l^{\ast}(S^{d-1})= \binom{d+l-1}{l}$. \\
	{\rm (iv)} $\dim W_l(\mathbb{R}^d)=\binom{d+l}{l}+\binom{d+l-1}{l-1}$. 
	\end{theorem}
We prove several lemmas that will be needed later.  
	\begin{lemma} \label{lem}
	Let $\mathcal{P}(\mathbb{R}^d)$ be a linear subspace of $P_l(\mathbb{R}^d)$, and $X$ a subset of $\mathbb{R}^d$. 
	Let $p_1, p_2, \ldots, p_m \in \mathcal{P}(\mathbb{R}^d)$, and $x_1, x_2, \ldots, x_n \in X$. Let $M$ be the $m \times n$ matrix whose $(i,j)$-entry is $p_i(x_j)$. Then the rank of $M$ is at most $\dim \mathcal{P}(X)$.     
	\end{lemma}
	\begin{proof}
	If the rank of $M$ is more than $\dim \mathcal{P}(X)$, then we have more than $\dim \mathcal{P}(X)$ linearly independent polynomials on $X$, it is a contradiction.
	\end{proof}
	
	\begin{lemma} \label{k_bound}
	Let $M$ be a symmetric matrix of size $n$. Let $e$ be an eigenvalue of multiplicity at least $m$. 
	If the diagonal entries of $M$ are all $0$ and the non-diagonal entries are $0$ or $\pm 1$, then $e^2 \leq (n-1)(n-m)/m$. 
	\end{lemma}
	\begin{proof}
	Let $a_1,a_2, \ldots, a_n$ be the eigenvalues of $M$. 
	We put $a_{n-m+1}=e, a_{n-m+2}=e, \ldots, a_{n}=e$.
	Since the sum of eigenvalues of $M$ is the trace of $M$, and the sum of the squares of the eigenvalues is the trace of $M^2$, we have
	\[
	a_1+a_2+\cdots +a_{n-m}+m e =0,
	\]
	\[
	a_1^2+a_2^2+\cdots +a_{n-m}^2+m e^2 \leq n(n-1). 
	\]
	By the Cauchy-Schwartz inequality, 
	\begin{align*}
		m^2 e^2 &= (\sum_{l=1}^{n-m} a_l)^2 \\
								&\leq (\sum_{l=1}^{n-m} {a_l}^2)(\sum_{l=1}^{n-m} 1^2) \\
								&\leq (n-m) \biggl(n(n-1) - m e^2 \biggr).
	\end{align*}
	This implies $e^2 \leq (n-1)(n-m)/m$.
	\end{proof}
	Let $I$ be the identity matrix. 
	\begin{lemma} \label{integer}
	Let $M$ be a symmetric matrix whose diagonal entries are all $k\in \mathbb{R}$ 
	and non-diagonal entries are integers. 
	Suppose the multiplicity of the zero eigenvalue of $M$ is greater than the maximum multiplicity of
 the non-zero eigenvalues of $M$. 
	Then $k$ is an integer. 
	\end{lemma}
	
	\begin{proof}
	We can write 
	$M =k I + A$,
	where  $A$ is a symmetric matrix whose diagonal entries are $0$ 
and non-diagonal entries are integers. 
Let $m$ be the multiplicity of the zero eigenvalue of $M$. 
Then $-k$ is a non-zero eigenvalue of $A$ of multiplicity $m$. 
Since the diagonal entries of $A$ are $0$ and the non-diagonal entries are integers, 
$-k$ is a real algebraic integer. 
Suppose $-k$ is irrational. 
The characteristic polynomial of $A$ is divisible by $f^m$ where $f$ is the minimal polynomial of $-k$. 
Hence, $-k$ has at least one conjugate $k' \in \mathbb{R}$ as an eigenvalue of $A$ of multiplicity $m$. 
Therefore $k+k'$ is a non-zero eigenvalue of $M$ of multiplicity $m$. This contradicts the fact
 that
$m$ is greater than the maximum multiplicity of the non-zero eigenvalues of $M$. 
Thus $-k$ is a rational algebraic integer, and hence $k$ is an integer.  
	\end{proof}
	
Throughout this paper, we use the function 
\[
U(N):=\left\lfloor \frac{1}{2}+ \sqrt{\frac{N^2}{2N-2}+\frac{1}{4}} \right\rfloor.
\]
The following is a key lemma to prove the main results of this paper. 
	\begin{lemma} \label{key lemma}
	Let $X$ be a finite subset of $\Omega \subset \mathbb{R}^d$, and $\mathcal{P}(\mathbb{R}^d)$ a 
linear subspace of $P_l(\mathbb{R}^d)$. Let $N=\dim \mathcal{P}(\Omega)$. Suppose that there exists 
$F_x \in \mathcal{P}(\Omega)$ for each $x \in X$ such that $F_x(x)=k$ where $k$ is constant, $F_x(y)$ are 
$0$ or $1$ for all $y \in X$ with $y \ne x$, and $F_x(y)=F_y(x)$ for all $x,y \in X$. If $|X| \geq 2 N$, 
then $k$ is an integer, and $|k|\leq U(N)$.
	
	\end{lemma}
	\begin{proof}
	Let $M$ be the symmetric matrix $(F_x(y))_{x,y \in X}$. 
By Lemma \ref{lem}, the rank of $M$ is at most $N$. 
Since $|X| \geq 2N$, the multiplicity of the zero eigenvalue of $M$ is at least $N$. 
We can express
	\begin{equation} \label{eq}
	M =k I + A,
	\end{equation}
	where $A$ is a symmetric $(0,1)$-matrix whose diagonal entries are $0$. 
The matrix $A$ is regarded as the 
adjacency matrix of a simple graph 
({\it i.e.} the matrix with rows and columns indexed by the vertices, such that $(u,v)$-entry 
is equal to $1$ if $u$ is adjacent to $v$, and other entries are zero). 

If there exists a connected 
component whose  diameter ({\it i.e.} longest shortest path) is at least $2$, then $A$ has at least $3$ 
distinct eigenvalues \cite[Lemma 8.12.1, page 186]{Godsil-Royle}. Then the multiplicity of a nonzero 
eigenvalue of $M$ is at most $N-1$. By Lemma \ref{integer}, $k$ is an integer.  

If every connected 
component is of diameter $1$ or an isolated vertex, then the eigenvalues of $A$ are integers. 
Since $-k$ is an eigenvalue of $A$, $k$ is an integer.

	By (\ref{eq}), we have the equality 
		\[
		2 A-(J-I)=2  M-J-(2 k-1) I,
		\]
	where $J$ is the matrix of all ones. Let $D=2 A-(J-I)$. The diagonal entries of $D$ are $0$, and its 
non-diagonal entries are $1$ or $-1$. Since the rank of $2  M-J$ is at most $N+1$, the multiplicity of the 
zero eigenvalue of $2  M-J$ is at least $|X|-N-1$. Therefore $-2 k+1$ is an eigenvalue of $D$ of 
multiplicity at least $|X|-N-1$. 
	By Lemma \ref{k_bound}, 
		$(2 k-1) ^2 \leq (N+1)(|X|-1)/(|X|-N-1)$. 
	Since $|X| \geq 2N$, we have 
		\begin{align*}
		\frac{(N+1)(|X|-1)}{|X|-N-1} & = N+1+\frac{N(N+1)}{|X|-N-1} \\
		& \leq \frac{2N^2}{N-1}+1. 
		\end{align*}
	This implies the second statement.
	\end{proof}
\section{Euclidean case}
We define $||\xi||=\sqrt{\sum_{i=1}^d \xi_i^2}$ where $\xi=(\xi_1,\xi_2,\ldots \xi_d) \in \mathbb{R}^d$.
The following is the main theorem of the present paper. 
	\begin{theorem} \label{main}
	Let $X$ be an $s$-distance set in $\mathbb{R}^d$ with $s \geq 2$, and $A(X)=\{\alpha_1, \alpha_2, \ldots, \alpha_s \}$. Let $N=\binom{d+s-1}{s-1}+\binom{d+s-2}{s-2}$.
	If $|X| \geq 2 N$, then  
\[
\prod_{j=1,2,\ldots,s, j \ne i} \frac{\alpha_j^2}{\alpha_j^2-\alpha_i^2}
\] 
is an integer $k_i$ for each $i=1,2,\ldots,s$. Furthermore $|k_i| \leq U(N)$.
	\end{theorem}  
	\begin{proof}
	We fix $i \in \{1,2,\ldots, s \}$. For each $x \in X$, we define the polynomial 
		\[
		F_x (\xi)=  \prod_{j=1,2,\ldots,s, j \ne i} \frac{\alpha_j^2-||x-\xi||^2  }{\alpha_j^2-\alpha_i^2}.
		\]
	Then, $F_x \in W_{s-1}(\mathbb{R}^d)$, $F_x(x)=\prod_{j \ne i} {\alpha_j}^2/({\alpha_j}^2-{\alpha_i}^2)$, $F_x(y)=1$ if $d(x,y)=\alpha_i$, $F_x(y)=0$ if $d(x,y) \ne \alpha_i$, and $F_x(y)=F_y(x)$ for all $x,y\in X$. By Lemma \ref{key lemma}, the theorem follows. 
	\end{proof} 
	\begin{remark}
For $s=2$ the upper bound for $|k_i|$ in Theorem \ref{main} is worse than that in the original theorem of Larman--Rogers--Seidel. 
\end{remark}
\begin{remark}
If the dimension of the linear space spanned by $\{F_x\}_{x \in X}$ is 
smaller than $\dim W_{s-1}(\mathbb{R}^d)$, then we can make the condition of $|X|$ be stronger. 
\end{remark}
	
	\begin{example} \label{exam}
	Let $X_{d,s}$ be the set of all vectors those are of length $d+1$, 
and have exactly $s$ entries of $1$ and $d+1-s$ entries of $0$. For any $x \in X_{d,s}$, the usual inner product of $x$ and the vector of 
all ones is equal to $s$. 
If $s\leq (d+1)/2$, the set $X_{d,s}$ can be regarded as 
a $\binom{d+1}{s}$--point $s$-distance set in $\mathbb{R}^d$. 
The inequality $\binom{d+1}{s} \geq 2 N$ holds for every sufficiently large $d$. 
For instance, for $s=3$ and $d \geq 10$, the inequality holds. 
	\end{example}
\section{The number of Euclidean $s$-distance sets} 
	In this section, we prove that there are finitely many $s$-distance sets $X$ in $\mathbb{R}^d$ 
with $|X|\geq 2\binom{d+s-1}{s-1}+2\binom{d+s-2}{s-2}$. 

Let $D:=\{(t_1,t_2,\ldots,t_{s-1}) \in \mathbb{R}^{s-1} \mid 0<t_1<t_2<\cdots < t_{s-1} < 1 \}$. 
For each $i \in \{1,2,\ldots,s\}$, $K_i$ is the function from $D$ to $\mathbb{R}$ 
defined by
\[
K_i(t_1,t_2,\ldots,t_{s-1}):= \prod_{j=1,2,\ldots, s, j \ne i} \frac{t_j}{t_j-t_i}, 
\]
where $t_s=1$. 
It is easy to prove that the equality $\sum_{i=1}^s K_i=1$ holds.
		Let $F$ be a function from $D$ 
to $(K_1,K_2,\ldots, K_{s-1})$. 
If the Jacobian of any point in $D$ is not equal to zero, then 
there exists the inverse function $F^{-1}$ by the inverse function theorem. 
This means that we can uniquely determine the distances $\alpha_i$ for given integers $k_i$ in Theorem \ref{main}.
	
	\begin{lemma} \label{jacob}
	Let $F(t_1,t_2,\ldots, t_{s-1}) = (K_1,K_2, \ldots, K_{s-1})$.
	Let $J$ be the Jacobian matrix of $F$. Then 
	\[
	\det(J)= (s-1)! \prod_{i=1}^{s-1} \frac{K_i}{t_s-t_i}.
	\] 
	\end{lemma}
	\begin{proof}
	Let $m_{i,j}=1/(t_i-t_j)$. 
	By direct calculations, we have 
	\[
	\frac{\partial{K_i}}{\partial t_j}= \begin{cases}
	\sum_{k=1,\ldots,s, k \ne i} m_{k,i} K_i \text{ if $i=j$,} \\
	\frac{t_i}{t_j}m_{i,j} K_i \text{ if $i \ne j$.}
	\end{cases}
	\]
	Therefore $\det(J)=(\prod_{i=1}^{s-1} K_i) \det(M) $, where $M(i,j)=\sum_{k=1,\ldots,s, k \ne i}m_{k,i}$ if $i=j$,
 and  $M(i,j)=t_i m_{i,j}/t_j$ if $i \ne j$. 
	Note that 
	\[
	M= \diag (t_1,t_2, \ldots ,t_{s-1}) M' \diag(t_1,t_2,\ldots ,t_{s-1})^{-1}
	\]
	where $M'(i,j)=\sum_{k=1,\ldots,s, k \ne i}m_{k,i}$ if $i=j$, 
and  $M'(i,j)=m_{i,j}$ if $i \ne j$. Thus, $\det(M)=\det(M')$. 
Using the common denominator for all terms, we can write 
\[
\det(M')= \frac{P(t_1,\ldots,t_{s-1})}{\prod_{j=1}^{s-1}(t_s-t_j) \prod_{i<j<s}(t_i-t_j)^2}.
\]
 where $P(t_1,\ldots,t_{s-1})$ is a polynomial of degree at most $(s-1)(s-2)$. 
 
 Fix $i$ and $j$ with $i\ne j$. Adding each column except for the column $i$ of $M'$ to the column $i$, the column $i$ changes to 
$(m_{s,1},m_{s,2},\ldots,m_{s,s-1})$. Multiplying the column $j$ of $M'$ by $(t_i-t_j)$,

\begin{multline} \label{det}
(t_i-t_j) \det({}^t M') \\ =\det \bordermatrix{
 &                         &                        &        &  i    &        &  j    &       &           \cr
 &                         &                        &        &\vdots &        &\vdots &       &           \cr
i&m_{s,1}                  &m_{s,2}	                & \cdots &m_{s,i}& \cdots &m_{s,j}&\cdots & m_{s,s-1}  \cr
 &\vdots                   & \vdots                 &        &\vdots &        &\vdots &       &  \vdots   \cr
j&\frac{m_{1,j}}{m_{i,j}}   & \frac{m_{2,j}}{m_{i,j}} & \cdots &   1   & \cdots & 1 +\sum_{k=1,\ldots,s,k\ne i,j} \frac{m_{k,j}}{m_{i,j}}  & \cdots & \frac{m_{s-1,j}}{m_{i,j}}  \cr
 &                         &                        &        &\vdots &        &\vdots &       &           \cr
},
\end{multline}
where other entries of the right hand side of (\ref{det}) are the same as those of ${}^tM'$. 
When $t_i=t_j$, the right hand side of (\ref{det}) is defined, and the column $i$ coincides with the column $j$. 
Thus, when $t_i=t_j$, we have $(t_i-t_j) \det(M')=0$. 
This means that $P(t_1,\ldots,t_{s-1})$ has the factor $(t_i-t_j)^2$ for any $i$ and $j$ with $i < j<s$. 
Since the degree of $P(t_1,\ldots,t_{s-1})$ is at most $(s-1)(s-2)$, we have 
$P(t_1,\ldots,t_{s-1}) = c \prod_{i<j<s} (t_i-t_j)^2$,   
 where $c$ is constant.  Thus, $\det(M')=c \prod_{j=1}^{s-1}m_{s,j}$ and hence $\prod_{j=1}^{s-1}(t_s-t_j) \det(M')=c$. 
Multiplying the column $i$ of $M'$ by $t_s-t_i$, we have 
\begin{align*}
&\prod_{j=1}^{s-1}(t_s-t_j) \det(M') \\
&= 
\left|
\begin{array}{cccc}
1+\sum_{k=2,\ldots,s-1}\frac{m_{k,1}}{m_{s,1}} & \frac{m_{1,2}}{m_{s,2}} & \cdots & \frac{m_{1,s-1}}{m_{s,s-1}}  \\
\frac{m_{2,1}}{m_{s,1}} & 1+\sum_{k=1,\ldots,s-1, k\ne 2}\frac{m_{k,2}}{m_{s,2}} & \cdots & \frac{m_{2,s-1}}{m_{s,s-1}} \\
\vdots & \vdots & & \vdots \\
\frac{m_{s-1,1}}{m_{s,1}}& \frac{m_{s-1,2}}{m_{s,2}} & \cdots & 1+\sum_{k=1,\ldots,s-2}\frac{m_{k,s-1}}{m_{s,s-1}}
\end{array}
\right| \\
& \rightarrow
\left|
\begin{array}{cccc}
1& 1 & \cdots & 1  \\
0 & 2+\sum_{k=3,\ldots,s-1}\frac{m_{k,2}}{m_{s,2}} & \cdots & \frac{m_{2,s-1}}{m_{s,s-1}} \\
\vdots & \vdots & & \vdots \\
0& \frac{m_{s-1,2}}{m_{s,2}} & \cdots & 2+\sum_{k=2,\ldots,s-2}\frac{m_{k,s-1}}{m_{s,s-1}}
\end{array}
\right| \text{ as $t_1 \rightarrow t_s$} \\
& \rightarrow
\left|
\begin{array}{ccccc}
1      &   1   & 1                                             & \cdots & 1  \\
0      &   2   & 1                                             & \cdots & 1 \\
0      &   0   & 3+\sum_{k=4,\ldots,s-1}\frac{m_{k,3}}{m_{s,3}}& \cdots & \frac{m_{3,s-1}}{m_{s,s-1}} \\
\vdots &\vdots & \vdots                                        &        & \vdots \\
0      &   0   & \frac{m_{s-1,3}}{m_{s,3}}                     & \cdots & 3+\sum_{k=3,\ldots,s-2}\frac{m_{k,s-1}}{m_{s,s-1}}
\end{array}
\right|  \text{ as $t_2 \rightarrow t_s$} \\
& \vdots \\
& \rightarrow
\left|
\begin{array}{ccccc}
1& 1 & 1&\cdots & 1 \\
0 & 2 & 1 &\cdots & 1 \\
\vdots &0 &3 &\cdots & 1 \\
\vdots & \vdots & \vdots& & \vdots \\
0 & 0 & 0 & \cdots & s-1
\end{array}
\right| \text{ as $t_{s-1} \rightarrow t_s$} \\
& = (s-1)!
\end{align*}
Therefore, we have $c=(s-1)!$, and hence this lemma follows. 
	\end{proof}
	
	\begin{lemma} \label{dis_uni}
	Let $X$ be an $s$-distance set with  $|X|\geq 2\binom{d+s-1}{s-1}+2\binom{d+s-2}{s-2}$, and 
$A(X)=\{ \alpha_1, \alpha_2, \ldots, \alpha_s=1 \}$. Suppose $k_i$ are the ratios in Theorem \ref{main}. 
 Then the distances $\alpha_i$ are uniquely determined from given integers $k_i$. 
	\end{lemma}
	\begin{proof}
	This lemma is straightforward from Lemma \ref{jacob}. 
	\end{proof}
	\begin{example}
For $s=3$, 
	\begin{align*}
	\alpha_1^2&= \frac{k_1(-1+k_1 +k_2) +\sqrt{k_1 k_2 (-1+k_1+k_2)}}{k_1 (k_1+k_2)},\\
	\alpha_2^2&=\frac{k_2(-1+k_1 +k_2)+\sqrt{k_1 k_2 (-1+k_1+k_2)}}{k_2 (k_1+k_2)}, 
	\end{align*}
where $k_1>0$ and $k_2<0$. 

\end{example}
	\begin{theorem} \label{finite_thm}
	There are finitely many $s$-distance sets $X$ in $\mathbb{R}^d$ 
with $|X|\geq 2\binom{d+s-1}{s-1}+2\binom{d+s-2}{s-2}$.
	\end{theorem}
	
	\begin{proof}
	By Theorem \ref{main}, we have finitely many pairs of integers $k_1,k_2,\ldots, k_{s-1}$. 
	Therefore, we have finitely many possible pairs of distances $\alpha_1,\alpha_2,\ldots,\alpha_{s-1}$ by Lemma \ref{dis_uni}. 
	For the finitely many pairs of distances, we can make finitely many distance matrices. 
	If there exists $X \subset \mathbb{R}^d$ such that $C=\{ d(x,y)^2\}_{x, y \in X}$ for a given distance matrix $C$, 
	then the finite set is unique up to congruences \cite{Neumaier}. 
	Therefore this theorem follows. 
	\end{proof}
	
	\begin{remark}
	Einhorn and Schoenberg \cite{Einhorn-Schoenberg} proved that there are finitely many two-distance
sets $X$ in $\mathbb{R}^d$ with $|X|\geq d+2$.  
	The condition $|X| \geq d+2$ is best possible because there are infinitely many $(d+1)$--point two-distance sets. 
	The inequality for $|X|$ in Theorem \ref{finite_thm} is not sharp even for $s=2$, hence it must be improved. 
	\end{remark}

\section{Spherical case}
In this section, we discuss $s$-distance sets on the unit sphere $S^{d-1}$. 
Let $B(X):=\{(x,y) \mid x,y \in X, x \ne y \}$, where $(,)$ means the usual inner product in $\mathbb{R}^d$. 

	The following is a generalization of Larman--Rogers--Seidel's theorem for spherical $s$-distance sets. 
	The condition of $|X|$ is stronger than that in Theorem \ref{main}. 
	\begin{theorem} \label{sphere}
	Let $X$ be an $s$-distance set in $S^{d-1}$ with $s \geq 2$, and $B(X)=\{\beta_1, \beta_2, \ldots, \beta_s \}$. Let 
$N=\binom{d+s-2}{s-1}+\binom{d+s-3}{s-2}$.
	If $|X| \geq 2 N$, then 
\[
\prod_{j=1,2,\ldots,s, j \ne i} \frac{1-\beta_j}{\beta_i-\beta_j}
\] is an integer $k_i$  for each $i=1,2,\ldots, s$. 
Furthermore $|k_i|\leq U(N)$.  
	\end{theorem}  
	
	\begin{proof}
	 We fix $i \in \{1,2,\ldots, s\}$. For each $x \in X$, we define the polynomial 
		\[
		F_x (\xi)= \prod_{j=1,2,\ldots,s, j \ne i} \frac{(x,\xi)-\beta_j  }{\beta_i-\beta_j}. 
		\]
	Then $F_x \in P_{s-1}(S^{d-1})$, $F_x(x)=\prod_{j \ne i}(1-\beta_j)/(\beta_i-\beta_j)$, 
and $\{F_x\}$ satisfies the condition in Lemma \ref{key lemma}. Hence this theorem follows. 
	\end{proof}
 A finite $X\subset S^{d-1}$ is said to be antipodal if $-x \in X$ for any $x \in X$. 
 Let $Y_X$ denote a subset of an antipodal set $X$ such that $Y_X \cup (-Y_X) = X$ and $|Y_X|=|X|/2$. 
 If $X$ is an antipodal spherical $s$-distance set, then $Y_X$ is an $(s-1)$-distance set. 
In the following theorems for antipodal spherical $s$-distance sets, 
the inequality for $|X|$ is not restrictive in comparison with that in Theorem \ref{sphere}.

	\begin{theorem}\label{anti1}
	Let $X$ be an antipodal $s$-distance set in $S^{d-1}$ where $s$ is an odd integer at least $5$. 
	Let $ B(X) = \{-1, \pm \beta_1, \pm \beta_2, \ldots, \pm \beta_{\frac{s-1}{2}} \}$. \\ 
$(1)$ Let $N=\binom{d+s-4}{s-3}$.
	If $|X| \geq 4 N$, then  
\begin{equation} \label{anti1(1)}
\prod_{j=1,2,\ldots, \frac{s-1}{2}, j \ne i}\frac{1-\beta_j^2}{\beta_i^2-\beta_j^2}
\end{equation}  
is an integer $k_i$ for each $i=1,2,\ldots, (s-1)/2$. Furthermore $|k_i| \leq U(N)$. \\
$(2)$ Let $N=\binom{d+s-3}{s-2}$.
If $|X| \geq 4 N+2$, then  
\begin{equation} \label{anti1(2)}
\frac{1}{\beta_i} \prod_{j=1,2,\ldots, \frac{s-1}{2}, j \ne i}\frac{1-\beta_j^2}{\beta_i^2-\beta_j^2}
\end{equation}  
is an integer $k_i$ for each $i=1,2,\ldots, (s-1)/2$. Furthermore $|k_i| \leq \lfloor \sqrt{2N^2/(N+1)} \rfloor$. \\
	\end{theorem} 
	\begin{proof}
$(1)$: We fix $i \in \{1,2,\ldots, (s-1)/2 \}$. For each $x \in Y_X$, we define the polynomial 
		\[
		F_x (\xi)= \prod_{j=1,2,\ldots, \frac{s-1}{2},j \ne i} \frac{(x,\xi)^2-\beta_j^2  }{\beta_i^2-\beta_j^2}.
		\]
	Then $F_x \in P_{s-3}^{\ast}(S^{d-1})$, $F_x(x)= \prod_{j \ne i}(1-\beta_j^2)/(\beta_i^2-\beta_j^2)$, and $\{F_x\}$ satisfies the condition in Lemma \ref{key lemma}. Hence $(1)$ follows. \\
$(2)$:  For each $x \in Y_X$, we define the polynomial 
		\[
		F_x (\xi)= \frac{(x,\xi)}{\beta_i} \prod_{j=1,2,\ldots, \frac{s-1}{2},j \ne i} \frac{(x,\xi)^2-\beta_j^2  }{\beta_i^2-\beta_j^2}.
		\]
	Then $F_x \in P_{s-2}^{\ast}(S^{d-1})$, 
$F_x(x)=1/\beta_i \prod_{j \ne i}(1-\beta_j^2)/(\beta_i^2-\beta_j^2)$, 
$F_x(y)=\pm 1$ if $(x,y)=\pm \beta_i$, $F_x(y)=0$ if $(x,y) \ne \pm \beta_i$, 
and $F_x(y)=F_y(x)$ for all $x,y\in Y_X$. 
Let $M$ be the symmetric matrix $(F_x(y))_{x,y \in Y_X}$. 
Note $|Y_X|\geq 2N+1$.
Since the rank of $M$ is at most $N$, the multiplicity of a non-zero eigenvalue is at most $N$, 
and the multiplicity of the zero eigenvalue is at least $N+1$. By Lemma \ref{integer}, 
$1/\beta_i \prod_{j \ne i}(1-\beta_j^2)/(\beta_i^2-\beta_j^2)$ is an integer $k_i$ for any $i$. 

Note that $-k_i$ is eigenvalue of $A:=M-k_i I$ of multiplicity at least $|Y_X|-N$. 
By Lemma \ref{k_bound} and $|Y_X|\geq 2N+1$, we have 
	\[
	k_i^2 \leq \frac{N(|Y_X|-1)}{|Y_X|-N}= N+ \frac{N(N-1)}{|Y_X|-N} \leq \frac{2 N^2}{N+1}.
	\] 
	Therefore $(2)$ follows. 
	\end{proof}

	\begin{theorem} \label{anti2}
	Let $X$ be an antipodal $s$-distance set in $S^{d-1}$ where $s$ is an even integer at least $4$. 
	Let $ B(X) = \{-1,\beta_1=0, \pm \beta_2, \ldots, \pm \beta_{\frac{s}{2}} \}$. \\
 $(1)$ Let $N=\binom{d+s-3}{s-2}$.
	If $|X| \geq 4 N$, then 
\begin{equation} \label{anti2(1)}
 \prod_{j=1,2,\ldots, \frac{s}{2}, j \ne i}\frac{1-\beta_j^2}{\beta_i^2-\beta_j^2}
\end{equation}  
is an integer $k_i$ for each $i=1,2,\ldots, s/2 $. Furthermore $|k_i| \leq U(N)$. \\
 $(2)$ Let $N=\binom{d+s-4}{s-3}$.
	If $|X| \geq 4 N+2$, then 
\begin{equation} \label{anti2(2)}
 \frac{1}{\beta_i} \prod_{j=2,3,\ldots, \frac{s}{2}, j \ne i}\frac{1-\beta_j^2}{\beta_i^2-\beta_j^2}
\end{equation}  
is an integer $k_i$ for each $i=2,3,\ldots, s/2 $. 
Furthermore $|k_i| \leq \lfloor \sqrt{2N^2/(N+1)} \rfloor$. 
	\end{theorem} 
	\begin{proof}
$(1)$: We fix $i \in \{1,2,\ldots, s/2 \}$. For each $x \in Y_X$, we define the polynomial 
		\[
		F_x (\xi)= \prod_{j=1,2,\ldots, \frac{s}{2},j \ne i} \frac{(x,\xi)^2-\beta_j^2  }{\beta_i^2-\beta_j^2}.
		\]
	Then $F_x \in P_{s-2}^{\ast}(S^{d-1})$, $F_x(x)= \prod_{j \ne i}(1-\beta_j^2)/(\beta_i^2-\beta_j^2)$, and $\{F_x\}$ satisfies the condition in Lemma \ref{key lemma}. Hence $(1)$ follows. \\
$(2)$:  For each $x \in Y_X$, we define the polynomial 
		\[
		F_x (\xi)= \frac{(x,\xi)}{\beta_i} \prod_{j=2,3,\ldots, \frac{s}{2} ,j \ne i} \frac{(x,\xi)^2-\beta_j^2  }{\beta_i^2-\beta_j^2}.
		\]
	Then $F_x \in P_{s-3}^{\ast}(S^{d-1})$. 
    By the same manner as the proof of Theorem \ref{anti1} (2), this proof is complete.   
	\end{proof}
	By the above theorems, we show the rationality of the inner products for
 a large antipodal $s$-distance set in $S^{d-1}$.  
	\begin{theorem} \label{rational}
	Suppose $X$ is an antipodal $s$-distance set in $S^{d-1}$ with $s\geq 4$.  
Suppose $|X| \geq 4 \binom{d+s-3}{s-2}+2$. 
Then $\beta$ is rational for any $\beta \in B(X)$. 
	\end{theorem}
	\begin{proof}
	By assumption, the values (\ref{anti1(1)}), (\ref{anti1(2)}), (\ref{anti2(1)}), and (\ref{anti2(2)}) are
	integers. 
	Dividing (\ref{anti1(1)}) by (\ref{anti1(2)}), or (\ref{anti2(2)}) by (\ref{anti2(1)}), 
we have every inner product $\beta_i$ is a rational number.  
	\end{proof}
	
	\begin{remark}
	Bannai--Damerell \cite{Bannai-Damerell1, Bannai-Damerell2} proved the result about the non-existence of tight spherical designs. 
	They showed the rationality of inner products of the finite set by the theory of 
	association schemes. The rationality played an important role to prove the non-existence of tight designs. 
	Theorem \ref{rational} shows another proof of the rationality of the inner products in a tight spherical 
	$(2s-1)$-design for sufficiently large $d$. 
	The rationality of inner products might be useful for a classification problem. 
	\end{remark}
	Finally remark that by the same method in the present paper, we can obtain similar theorems to  
Theorems \ref{main} or \ref{sphere} for spherical polynomial spaces \cite[Chapter 14]{Godsil} (for example, 
the Johnson scheme, or the Hamming scheme). 

	\quad
	
	\noindent
	 \textbf{Acknowledgments.} 
This research has been done during the stay at the
University of Texas at Brownsville, under the sponsorship of the Japan Society for the
Promotion of Science. The encouragement of Oleg Musin, who was the host of the stay, was invaluable. 
The author also thanks to Eiichi Bannai, Alexander Barg, 
Akihiro Munemasa, Makoto Tagami and Masashi Shinohara for useful discussions and comments.


\begin{thebibliography}{9}

	\bibitem{Bannai-Bannai-1}
	E. Bannai and E. Bannai, 
	An upper bound for the cardinality of an $s$-distance subset in real Euclidean space. 
	{\it Combinatorica} 1 (1981), no.\ 2, 99--102.
	
	
\bibitem{Bannai-Damerell1}
E. Bannai and R.M. Damerell, Tight spherical designs. I, {\it J.\ Math.\ Soc.\ Japan} 31 (1979), no.\ 1, 199--207.

\bibitem{Bannai-Damerell2}
E. Bannai and R.M. Damerell, Tight spherical designs. II, {\it J.\ London Math.\ Soc.\ }(2) 21 (1980), no.\ 1, 13--30. 

	
	\bibitem{Bannai-Bannai-Stanton}E. Bannai, E. Bannai, and D. Stanton, 
	An upper bound for the cardinality of an s-distance subset in real Euclidean space, II, 
	\textit{Combinatorica} 3 (1983), 147--152.
	
	\bibitem{Blokhuis}
	A. Blokhuis, 
	A new upper bound for the cardinality of $2$-distance sets in Euclidean space,  
	Convexity and graph theory (Jerusalem, 1981), 65--66, 
	{\it North-Holland Math. Stud.}, 87, North-Holland, Amsterdam, 1984. 
	
	\bibitem{Croft}H.\ T.\ Croft, 
	$9$-point and $7$-point configuration in $3$-space, 
	\textit{Proc.\ London.\ Math.\ Soc.\ }(3), 
	\textbf{12} (1962), 400--424. 
	
	\bibitem{Delsarte-Goethals-Seidel2}
	P. Delsarte, J.M. Goethals, and J.J. Seidel, 
	Bounds for systems of lines, and Jacobi polynomials, {\it Philips Res. Repts} 30, 91--105 (1975).
	
	\bibitem{Delsarte-Goethals-Seidel}
	P. Delsarte, J.M. Goethals, and J.J. Seidel, 
	Spherical codes and designs, {\it Geom.\ Dedicata} 6 (1977), no.\ 3, 363--388. 
	
	\bibitem{Einhorn-Schoenberg}
	S.J. Einhorn and I.J. Schoenberg, 
	On Euclidean sets having only two distances between points I, II, 
	{\it Indag.\ Math.}, 28 (1966), 479--488, 489--504.
	(Nederl.\ Acad.\ Wetensch.\ Proc.\ Ser.\ A69)
	
	
	
	\bibitem{Godsil}
	C.D. Godsil, 
	{\it Algebraic Combinatorics.} Chapman and Hall, New York, 1993.
	
	\bibitem{Godsil-Royle}
	C. Godsil and G. Royle, 
	{\it Algebraic Graph Theory}, Graduate Texts in Mathematics 207, New York: Springer-Verlag, 2001.
	
	
	\bibitem{Kelly}L.M. Kelly, 
	Elementary Problems and Solutions. Isosceles $n$-points, 
	\textit{Amer.\ Math.\ Monthly}, 
	\textbf{54} (1947), 227--229.
	
	\bibitem{Larman-Rogers-Seidel}D.G. Larman, C.A. Rogers, and J.J. Seidel, 
	On two-distance sets in Euclidean space, 
	\textit{Bull.\ London Math.\ Soc.} 
	9 (1977), 261--267.
	
	
	\bibitem{Lisonek}P. Lison\v{e}k, 
	New maximal two-distance sets, 
	\textit{J.\ Comb.\ Theory, Ser.\ A} 77 (1997), 318--338.

	
	\bibitem{Musin}
	O.R. Musin, On spherical two-distance sets, 
	{\it J.\ Comb.\ Theory Ser.\ A} 116 (4) (2009), 988--995. 

	
	\bibitem{Neumaier}
	A. Neumaier,
	Distance matrices, dimension, and conference graphs,
	\textit{Nederl.\ Akad.\ Wetensch.\ Indag.\ Math.}43 (1981), 385--391.
	
	
	
\end{thebibliography}
\end{document}